\documentclass[11pt,a4paper]{article}
 \usepackage{euscript}
\usepackage{amsmath}
\usepackage{graphicx}
\usepackage{amsthm}
\usepackage{amssymb}
\usepackage{epsfig}
\usepackage{url}
\usepackage{amsmath,amscd}
\theoremstyle{theorem}
\newtheorem{theorem}{Theorem}[section]

\newtheorem{lemma}[theorem]{Lemma}
\newtheorem{proposition}[theorem]{Proposition}

\numberwithin{equation}{section}
 
\title{Noncommutative Henselizations} 
\author{Masood Aryapoor\\
\tiny{\textit{Division of Mathematics and Physics}}\\
\tiny{\textit{M\"{a}lardalen  University}}\\
\tiny{\textit{Hamngatan 15, 632 17, Eskilstuna, 
Sweden
}}

}
 \date{}
\begin{document}
 \maketitle
\begin{abstract}
\noindent
In this paper, the familiar notion of a Henselian pair is extended to the noncommutative case. Furthermore, the problem of Henselizations is studied in the noncommutative context, and it is shown that every (not necessarily commutative) pair which is Hausdorff with respect to a certain topology has a left (and right) Henselization. 
\end{abstract}
\begin{section}{Introduction} 
The notion of a Henselian ring, introduced by Azumaya in \cite{azumaya1951maximally}, is well-known in the commutative case. This concept has been extended to the noncommutative case, see \cite{aryapoor2009non}. However, the theory of noncommutative Henselian rings is not as well-studied as its commutative counterpart, and many problems concerning noncommutative Henselian rings are still open. One such problem, discussed in \cite{aryapoor2009non}, is the problem of noncommutative Henselizations. In the commutative case, it is known that every local ring has a Henselization, see \cite{nagata1962local}.   
The aim of this article is to investigate the notion of Henselization in the noncommutative context. One of our results is that every noncommutative local ring satisfying a kind of ``commutativity'' condition has a Henselization, see Subsection \ref{localringHensel}.

In Section \ref{section2}, we present the preliminaries. In particular, the notion of a noncommutative Henselian ring, introduced in \cite{aryapoor2009non}, is generalized in two different ways. First, we introduce the notions of left Henselian rings and right Henselian rings, which is more natural in the noncommutative context. Second, this concept is generalized to pairs as done by Lafon in the commutative case, see \cite{lafon1963anneaux}. The final section is devoted to the notions of left Henselizations and right Henselizations.  It is shown that every perfect pair has a left (and right) Henselization which is unique up to unique isomorphism. 

\end{section} 
\begin{section}{Henselian pairs} \label{section2}
In this section, we present the preliminaries, and in particular, the notions of left Heneslian rings and right Henselian rings. In this paper, we shall assume that all rings have a unit, and all ring homomorphisms are unit-preserving. 
\begin{subsection}{The category of pairs}
The category of pairs has been introduced in the commutative setting, see \cite{lafon1963anneaux}. One can enlarge this category to include noncommutative rings. More precisely, the objects of the category $\mathcal{P}$ of pairs are pairs $(A,I)$ where $A$ is a (unitary but not necessarily commutative) ring and $I$ is an ideal of $A$. A morphism $\phi:(A,I)\to (B,J)$ of pairs is a ring homomorphism $\phi:A\to B$ satisfying  $\phi^{-1}(J)=I$. It is straightforward to verify that any directed system in $\mathcal{P}$ has a direct limit. Likewise, any inverse system in $\mathcal{P}$ has an inverse limit.

For a pair $(A,I)$, we denote the image of $a\in A$ in $A/I$ by $\bar{a}$. Given a morphism $\phi:(A,I)\to (B,J)$ of pairs, we have a canonical homomorphism $\bar{\phi}:A/I\to B/J$ induced by $\phi$. 
The proof of the following lemma is straightforward and left to the reader. 
\begin{lemma}\label{morphism}
Let  $\phi:(A,I)\to (C,K)$ and $\psi:(B,J)\to (C,K)$ be morphisms of pairs. If there is a ring homomorphism $\alpha:A\to B$ such that $\phi=\psi\circ\alpha$ as ring homomorphisms, then $\alpha:(A,I)\to (B,J)$ defines a morphism of pairs. 
\end{lemma}

\end{subsection}
\begin{subsection}{Coprime polynomials}
For a ring $A$, the ring of polynomials $F(x)=\sum_{i=0}^na_ix^i$ over $A$, where the indeterminate $x$ commutes with all $a\in A$, is denoted by $A[x]$. For a subset $S\subset A$, the set of all polynomials   $\sum_{i=0}^na_ix^i\in A[x]$, where $a_0,...,a_n\in S$, is denoted by $S[x]$.

Two polynomials $F_1(x),F_2(x)\in A[x]$ are called \textit{left coprime} if 
$$A[x]F_1(x)+A[x]F_2(x)=A[x].$$
The notion of right coprime polynomials is defined in the obvious way. It is easy to see that polynomials $F_1(x),F_2(x)\in A[x]$ are left coprime if and only if there are polynomials $G_1(x),G_2(x)\in A[x]$ such that $$G_1(x)F_1(x)+G_2(x)F_2(x)=1.$$
For a pair $(A,I)$, we denote the canonical homomorphism $A[x]\to (A/I)[x]$ by $F(x)=\sum_i a_i x^i\mapsto \bar{F}(x)=\sum_i \bar{a}_i x^i$. 
\end{subsection}
\begin{subsection}{Euclidean algorithm}
Let $A$ be  a ring. The following version of Euclidean algorithm holds in $A[x]$.
  
\begin{lemma}\label{EuclidAlg}
Let $I$ be a right ideal of $A$.  For every monic polynomial $F(x)\in A[x]$ and every polynomial $G(x)\in I[x]$, there exist polynomials $Q(x),R(x)\in I[x]$ such that 
$$G(x)=Q(x)F(x)+R(x) \; \text{and}\; \deg(R(x))<\deg(F(x)).$$  
\end{lemma}
\begin{proof}
We use induction on $\deg(G(x))$. If $\deg(G(x))<\deg(F(x))$, we set $Q(x)=0$ and $R(x)=G(x)$, so we are done. Let $G(x)=\sum_{i=0}^mb_ix^i$ and $$F(x)=a_0+a_1x+\cdots+a_{n-1}x^{n-1}+x^n,$$
where  $m\geq n$. The polynomial $G_1(x)=G(x)-b_m x^{m-n}F(x)$ belongs to $I$ because $b_m\in I$ and $I$ is a right ideal. Since 
$\deg(G_1(x))<\deg(G(x)),$ by induction, there are 
polynomials $Q_1(x),R(x)\in I[x]$ such that 
$$G_1(x)=Q_1(x)F(x)+R(x) \; \text{and}\; \deg(R(x))<\deg(F(x)).$$  
Setting $Q(x)=b_mx^{m-n}+Q_1(x)$, we obtain 
$$G(x)=Q(x)F(x)+R(x),$$
where $Q(x),R(x)$ satisfy the desired conditions, and we are done. 
\end{proof}

\end{subsection}
\begin{subsection}{Jacobson pairs}

A pair $(A,I)$ is called a \textit{Jacobson pair} if $I\subset rad(A)$ where $rad(A)$ is the Jacobson radical of $A$. We have the following result concerning Jacobson pairs. 
\begin{lemma}\label{coprime-reduction}
Let $(A,I)$ be a Jacobson pair. Then, monic polynomials $F_1(x),F_2(x)\in A[x]$ are left (resp. right) coprime if and only if the polynomials $\bar{F}_1(x),\bar{F}_2(x)\in (A/I)[x]$ are left (resp. right) coprime. 
\end{lemma}

\begin{proof}
The ``only if'' part is trivial. To prove the other direction, suppose that $\bar{F}_1(x),\bar{F}_2(x)\in (A/I)[x]$ are left coprime. It follows that $$A[x]F_1(x)+A[x]F_2(x)+I[x]=A[x]$$
Considering $$M=\frac{A[x]}{A[x]F_1(x)}$$ as a left $A$-module in the obvious way, we see that 
$M=N+IM$ where $N$ is the submodule of $M$ generated by $F_2(x)$.  
The module $M$ is a finitely generated $A$-module because $F_1(x)$ is monic. Since $I\subset rad(A)$, Nakayama's lemma
implies that $N=M$, that is,   
$$A[x]F_1(x)+A[x]F_2(x)=A[x],$$
and we are done. 
\end{proof}

\end{subsection}
\begin{subsection}{Local homomorphisms}
A ring homomorphism $\phi:A\to B$ is called \textit{local} if $\phi$ sends every nonunit in $A$ to a nonunit in $B$. Here, we provide some facts concerning local maps. The proof of the first lemma is easy and left to the reader. 
\begin{lemma}\label{localmaps}
Let $\phi:A\to B$ and $\psi:B\to C$ be ring homomorphisms. If $\psi\circ \phi$ is local, then $\phi$ is local too.
\end{lemma}

\begin{lemma}\label{localmapsLocalrings}
Let $\phi:A\to B$ be a local homomorphism. If $B$ is a local ring, then $A$ is a local ring whose maximal ideal is $\phi^{-1}(I)$ where $I$ is the maximal ideal of $B$. 
\end{lemma}

\begin{proof}
To show that $A$ is a local ring, we need to prove that if $a+b$ is invertible in $A$ for some $a,b\in A$, then either $a$ or $b$ is invertible in $A$, see Theorem 19.1 in \cite{lam2013first}. If $a+b$ is invertible in $A$ for some $a,b\in A$, then $\phi(a)+\phi(b)$ is invertible in $B$ which implies that $\phi(a)$ or $\phi(b)$ is invertible in $B$, because $B$ is local. It follows that either $a$ or $b$ is invertible in $A$ because $\phi$ is local. Therefore, $A$ is a local ring. Since $\phi$ is a local homomorphism and $B$ is a local ring, we see that $\phi^{-1}(I)$ consists of all nonunits in $A$, that is, $\phi^{-1}(I)$ is the maximal ideal of $A$. 
\end{proof}

\end{subsection}
\begin{subsection}{Localizations}\label{localizationsection}
Let $\phi:A\to B$ be a ring homomorphism. In what follows, we introduce a ring $A_\phi$, a ring homomorphism 
$\Lambda_\phi: A\to A_\phi$ and a local homomorphism $\Psi_\phi:  A_\phi\to B$ such that $\phi=\Psi_\phi \circ \Lambda_\phi$. Let $S_1$ be the set of all elements $s\in A$ such that $\psi_1(s)$ is invertible in $B$. Let $A_1=A_{S_1}$ be the localization of $A$ at $S_1$ and $\lambda_1:A\to A_{1}$ be the canonical homomorphism, consult \cite{cohn_1995}
for the definition and elementary properties of localizations. It follows from the universal property of $A_{S_1}$ that there exists a unique homomorphism $\psi_1:A_{1}\to B$ such that $\phi=\psi_1  \circ \lambda_1$. Let $S_2$ be the set of all elements $s\in A_1$ such that $\phi_1(s)$ is invertible in $B$. Consider the localization $A_2=(A_{1})_{S_2}$ of $A_1$ at $S_2$ and let $\lambda_2:A_1\to A_{2}$ be the canonical homomorphism. It follows from the universal property of $A_{2}$ that there exists a homomorphism $\psi_2:A_{2}\to B$ such that $\psi_1=\psi_2 \circ \lambda_2$. Continuing this process, we obtain a sequence of rings $A_1,A_2,A_3,...$, homomorphisms 
$\lambda_{i}:A_{i-1}\to A_{i}$ and $\psi_i:A_i\to B$ such that $\psi_{i-1}=\psi_{i}  \circ \lambda_i$ for $i=1,2,...,$ where $\psi_0=\phi$. We set $A_\phi=\varinjlim A_i$. We have a canonical homomorphism $\Lambda_\phi:A\to A_\phi$ and a canonical homomorphism 
$\Psi_\phi:  A_\phi\to B$. In the following proposition, we provide some facts about this construction, see Section 2 in \cite{aryapoor2010f} for a proof of this proposition and  a detailed discussion of the localization ring $A_\phi$.  
\begin{proposition}\label{localization}
The ring homomorphism $\Psi_\phi:  A_\phi\to B$ is local and satisfies the equality $\phi=\Psi_\phi  \circ\Lambda_\phi$.
\end{proposition} 

\end{subsection}
\begin{subsection}{Commutativity with respect to a filtration}
Let $(A,I)$ be a pair. A descending sequence $\mathcal{F}$ of ideals
$$I_1\supset I_2 \supset ...$$
of $A$ is called a \textit{filtration} on $(A,I)$ if $I_1=I$. 
The pair $(A,I)$ is called \textit{commutative with respect to} $\mathcal{F}$, or $\mathcal{F}$-\textit{commutative} for short, if 
$[A,I_n]\subset I_{n+1}$ for all $n=1,2,...$. The notation $[S,T]$, where $S,T\subset A$, stands for the set of all elements of the form $st-ts$ where $s\in S, t\in T$.   
\end{subsection}
\begin{subsection}{The commutator filtration}
For every pair $(A,I)$, one can define a sequence $I^{(1)}, I^{(2)},...$ of ideals of $A$ as follows: $I^{(1)}=I$; for $n=1,2,...$, the ideal $I^{(n+1)}$ is the ideal generated by $[A,I^{(n)}]$. It is easy to see that the sequence $I^{(1)}, I^{(2)},...$ is a filtration on $(A,I)$ called the \textit{commutator filtration} of $(A,I)$. Clearly, $(A,I)$ is commutative with respect to its commutator filtration. Note that if $\phi:(A,I)\to (B,J)$ is a morphism of pairs then $\phi(I^{(n)})\subset  J^{(n)}$ for all $n=1,2,...$. 
\end{subsection}
\begin{subsection}{Topologies defined by filtrations}

Let $(A,I)$ be a pair and $\mathcal{F}: I_1=I, I_2,...$ be a filtration on $(A,I)$. 
The $\mathcal{F}$-\textit{topology} on $(A,I)$ is the linear topology on $A$ for which the sets $I,I_2,\dots$ form a fundamental system of neighborhoods of $0$. The pair $(A,I)$ is called \textit{separated with respect to} $\mathcal{F}$, or  $\mathcal{F}$\textit{-separated}
for short, if $A$ is Hausdorff with respect to the $\mathcal{F}$-topology. Note that $(A,I)$  is $\mathcal{F}$-separated if and only if $\cap_{n=1}^\infty I_n =\{0\}$. 
The pair $(A,I)$ is called \textit{complete with respect to} $\mathcal{F}$, or $\mathcal{F}$\textit{-complete} for short, if it is complete with respect to the $\mathcal{F}$-topology. 
  
\end{subsection}
\begin{subsection}{Unique factorization pairs}

A pair $(A,I)$ is called a \textit{left unique factorization pair}, or  \textit{LUFP} for short, if for every 
factorization $\bar{F}(x)=f_1(x)f_2(x)$ of a monic polynomial $F(x)\in A[x]$  over $A/I$, where 
$f_1(x)f_2(x)\in (A/I)[x]$ are  left coprime monic polynomials, there exists at most one factorization $F(x)=F_1(x)F_2(x)$ such that $F_1(x), F_2(x)\in A[x]$ are monic polynomials, and 
$\bar{F}_1(x)=f_1(x), \bar{F}_2(x)=f_2(x)$. The notion of a right unique factorization pair (RUFP) is defined in a similar way. 
A pair is called a \textit{unique factorization pair}, or  \textit{UFP} for short, if it is both an LUFP and an RUFP. 

\end{subsection}
\begin{subsection}{Henselian pairs} 
A pair $(A,I)$ is called \textit{left Henselian} if $(A,I)$ is a Jacobson pair, and the following version of Hensel's lemma holds in A. For every monic polynomial $F(x)\in A[x]$, if $\bar{F}(x)=f_1(x)f_2(x)$, where $f_1(x),f_2(x)\in (A/I)[x]$ are left coprime monic polynomials, then there exist  unique monic polynomials
$F_1(x),F_2(x)\in A[x]$ satisfying $F(x)=F_1(x)F_2(x)$, $\bar{F}_1(x)=f_1(x)$ and $\bar{F}(x)=f_2(x)$. We note that  the polynomials $F_1(x)$ and $F_2(x)$ are left coprime, see Lemma \ref{coprime-reduction}. The notion of a right Henselian pair is defined in a similar fashion. A pair which is both left and right Henselian is called \textit{Henselian}. Obviously, every (resp. left or right) Henselian ring is a (resp. left or right) UFP.  We note that if $A/I$ is a commutative ring, then $(A,I)$ is left Henselian if and only if it is right Henselian. 
\end{subsection}
\begin{subsection}{A class of left Henselian pairs}
The following result generalizes Theorem 2.1 in \cite{aryapoor2009non}.
\begin{theorem}\label{HenselianRings}
Let $(A,I)$ be a pair and $\mathcal{F}: I_1=I,I_2,...$ be  a filtration on $(A,I)$  such that  $(A,I)$ is 
$\mathcal{F}$-commutative, $\mathcal{F}$-separated and $\mathcal{F}$-complete. If $I_n[A,A]\subset I_{n+1}$ and $I_n^2\subset I_{n+1}$ for all $n=1,2,...$, then $(A,I)$ is left Henselian.
\end{theorem}
\begin{proof}
First, we show that $(A,I)$ is a Jacobson pair. Let $a\in I$ be given. The condition $I_n^2\subset I_{n+1}$ for all $n=1,2,...$, implies that
the geometric series $$1-a+a^2-\cdots$$ converges to a unique limit in $A$ because $(A,I)$ is 
 $\mathcal{F}$-separated and $\mathcal{F}$-complete. The limit of this series is the inverse of $1+a$. Therefore, every element in $1+I$ is invertible, which implies that $I\subset rad(A)$, that is, $(A,I)$ is a Jacobson pair. 

To show that $(A,I)$ is left Henselian,
let $F(x)\in A[x]$ be a monic polynomial such that $\bar{F}(x)=f_1(x)f_2(x)$ where  $f_1(x),f_2(x)\in (A/I)[x]$
are left coprime monic polynomials. We need to show that we can lift this factorization to $A[x]$ in a unique way.     
First, we show that there are sequences of monic polynomials $$ F_{1,1}(x), F_{1,2}(x),..., F_{1,i}(x),... $$ and 
$$ F_{2,1}(x), F_{2,2}(x),..., F_{2,i}(x),... $$ in $A[x]$ such that 
$$\bar{ F}_{1,i}(x)=f_1(x),\;  \bar{ F}_{2,i}(x)=f_2(x),$$
 $$F_{1,i+1}(x)-F_{1,i}(x) \in I_{i}[x],\; F_{2,i+1}(x)-F_{2,i}(x) \in I_{i}[x],$$
$$F(x)-F_{1,i}(x)F_{2,i}(x)\in I_{i}[x],$$ for all $i\geq 1$. To construct these sequences, we use induction on $i$. 
Since the canonical map $A\to A/I$ is onto, we can find monic polynomials $$F_{1,1}(x), F_{2,1}(x)\in A[x]$$ such that $\bar{F}_{1,1}(x)=f_1(x), \bar{ F}_{2,1}(x)=f_2(x)$. Clearly,  $F_{1,1}(x)$ and  $F_{2,1}(x)$ satisfy the desired conditions. 
Having found $F_{1,i}(x)$ and $F_{2,i}(x)$, we find $F_{1,i+1}(x)$ and $F_{2,i+1}(x)$ as follows. Set 
$$G(x)=F(x)-F_{1,i}(x)F_{2,i}(x).$$
I claim that there are polynomials $R_1(x), R_2(x)\in I_i[x]$ such that  
$$\deg(R_1(x))<\deg(F_{2,i}(x)), \deg(R_2(x))<\deg(F_{1,i}(x)),$$ and 
$$R_2(x) F_{1,i}(x)+R_1(x) F_{2,i}(x)-G(x)\in I_{i+1}[x]$$
By Lemma  \ref{coprime-reduction},  there exit polynomials $H_1(x), H_2(x)\in A[x]$ such
$$H_1(x)F_{1,i}(x)+H_2(x)F_{2,i}(x)=1.$$
It follows that 
$$G_1(x)F_{1,i}(x)+G_2(x)F_{2,i}(x)=G(x),$$
where both $G_1(x)=G(x)H_1(x)$ and $G_2(x)=G(x)H_2(x)$ belong to $I_i[x]$. 
By Lemma \ref{EuclidAlg}, there are polynomials $Q(x),R_1(x)\in I_i[x]$ such that 
$$G_1(x)=Q(x)F_{2,i}(x)+R_1(x) \; \text{and}\; \deg(R_1(x))<\deg(F_{2,i}(x)).$$  
It follows that 
$$G(x)=R_1(x)F_{1,i}(x)+(G_2(x)+Q(x)F_{1,i}(x))F_{2,i}(x)+$$
$$Q(x)(F_{2,i}(x)F_{1,i}(x)-F_{1,i}(x)F_{2,i}(x))$$
Using the condition $I_i[A,A]\subset I_{i+1}$, we see that
$$ R_1(x)F_{1,i}(x)+(G_2(x)+Q(x)F_{1,i}(x))F_{2,i}(x)-G(x)\in I_{i+1}[x].$$
Let 
$G_2(x)+Q(x)F_{1,i}(x)=\sum_{i=0}^mb_ix^i.$
Assume that $m\geq \deg(F_{1,i}(x))$.
Since $F_{2,i}(x)$ is monic, and
$$\deg(R_1(x)F_{1,i}(x))< \deg(F_{1,i}(x))+\deg(F_{2,i}(x)),$$
$$\deg(G(x))<\deg(F_{1,i}(x))+\deg(F_{2,i}(x)),$$
the relation 
$$ R_1(x)F_{1,i}(x)+(G_2(x)+Q(x)F_{1,i}(x))F_{2,i}(x)-G(x)\in I_{i+1}[x]$$
implies that $b_m\in I_{i+1}$. Therefore, we have
$$ R_1(x)F_{1,i}(x)+(G_2(x)-b_mx^m)F_{2,i}(x)-G(x)\in I_{i+1}[x].$$
Using induction, we conclude that 
$$ R_1(x)F_{1,i}(x)+R_2(x)F_{2,i}(x)-G(x)\in I_{i+1}[x],$$
where $R_2(x)=\sum_{i<\deg(F_1(x))}b_ix^i$, proving the claim. We set
$$F_{1,i+1}(x)=F_{1,i}(x)+R_2(x),$$
$$F_{2,i+1}(x)=F_{2,i}(x)+R_1(x).$$
Clearly, we have
$$\bar{ F}_{1,i+1}(x)=f_1(x),\;  \bar{ F}_{2,i+1}(x)=f_2(x)$$
 $$F_{1,i+1}(x)-F_{1,i}(x) \in I_{i}[x],\; F_{2,i+1}(x)-F_{2,i}(x) \in I_{i}[x]$$
Moreover, we have
$$F(x)-F_{1,i+1}(x)F_{2,i+1}(x)=$$
$$(F(x)-F_{1,i}(x)F_{2,i}(x))-F_{1,i}(x)R_1(x)-R_2(x)F_{2,i}(x)-R_2(x)R_1(x)=$$
$$(G(x)-R_1(x)F_{1,i}(x)-R_2(x)F_{2,i}(x))+$$
$$(R_1(x)F_{1,i}(x)-F_{1,i}(x)R_1(x))-R_2(x)R_1(x).$$
Since $(A,I)$ is $\mathcal{F}$-commutative and $I_i^2\subset I_{i+1}$, we deduce that
$$F(x)-F_{1,i+1}(x)F_{2,i+1}(x)\in I_{i+1}[x].$$
Having constructed the desired sequences, we proceed as follows. Since $A$ is $\mathcal{F}$-complete,
the limits $$F_1(x)=\lim_{i\to\infty} F_{1,i}(x)\, \text{and}\, F_2(x)=\lim_{i\to\infty} F_{2,i}(x)$$ exist. Clearly, we have $\bar{ F}_{1}(x)=f_1(x)$ and $ \bar{ F}_{2}(x)=f_2(x)$. Since $A$ is $F$-Hausdorff, we have $F(x)=F_1(x)F_2(x)$ and $F_1,F_2$ are monic polynomials. Since $(A,I)$ is $\mathcal{F}$-commutative, it is easy to see that $I^{(n)}\subset I_n$ for all $n$. It follows that $(A,I)$ is also separated with respect to its commutator filtration. Therefore, 
by Proposition \ref{uniquenesslifting}, this factorization is unique, and we are done. 
\end{proof}
  
\end{subsection}
\end{section} 
\begin{section}{Noncommutative Henselizations}
The notion of the Henselization of a pair has been introduced in the commutative case, see \cite{lafon1963anneaux}. 
It turns out that one can develop a similar theory in the noncommutative case. However, we focus our attention on a special subcategory $\mathcal{P}_0$ of the category $\mathcal{P}$ of pairs and prove that every object in this subcategory has a left (and a right) Henselization in $\mathcal{P}_0$, see Theorem \ref{NHenselCom}. Our treatment of Henselization is somewhat similar to the one given in \cite{greco1969henselization}.
\begin{subsection}{The category of perfect pairs}

A pair $(A,I)$ is called \textit{perfect} if $(A,I)$ is separated with respect to its commutator filtration, that is, $
\cap_{n=1}^\infty I^{(n)}=\{0\}$. The full subcategory of the category $\mathcal{P}$ consisting of perfect pairs is denoted by $\mathcal{P}_0$. For any pair $(A,I)$, it is easy to see that the pair 
$$\digamma{(A,I)}=(\frac{A}{\cap_{n=1}^\infty I^{(n)}},\frac{I}{\cap_{n=1}^\infty I^{(n)}})$$
is a perfect pair. Furthermore, the assignment $(A,I)\mapsto \digamma{(A,I)}$ gives rise to a functor $\digamma:\mathcal{P}\to \mathcal{P}_0$. One can easily check that $\digamma$ is a left adjoint of the inclusion function $\iota:\mathcal{P}_0\to \mathcal{P}$. 

\end{subsection}
\begin{subsection}{Perfect Jacobson pairs}
The reason for restricting our attention to $\mathcal{P}_0$ is in the following result.

\begin{proposition}\label{uniquenesslifting}
Every perfect Jacobson pair is a UFP. 
\end{proposition}

\begin{proof}
Let $(A,I)$ be a perfect Jacobson pair. We only show that $(A,I)$ is an LUFP. The proof that A is an RUFP is similar. 
Assume, on the contrary, that there are different factorizations $$F(x)=F_1(x)F_2(x)=G_1(x)G_2(x)$$
of a monic polynomial $F(x)\in A[x]$ such that $f_1(x)=\bar{F}_1(x)=\bar{G}_1(x)$ and 
$f_2(x)=\bar{F}_2(x)=\bar{G}_2(x)$ are left coprime monic polynomials. Without loss of generality, we may assume  
$F_2(x)\neq G_2(x)$. The facts that $F_2(x)$ and $G_2(x)$ are monic polynomials, and $\bar{F}_2(x)=\bar{G}_2(x)$, imply that there exists a polynomial $N(x)\in I[x]$ such that $$F_2(x)=G_2(x)+N(x)\; \text{and} \; \deg(N(x))<\deg(G_2(x)).$$
Since $\cap_{n=1}^\infty I^{(n)}=\{0\}$ and $N(x)\neq 0$, there exists $d\geq 1$ such that 
$$N(x)\in I^{(d)}[x]\; \text{but} \; N(x)\notin I^{(d+1)}[x].$$ 
Since $\bar{F}_1(x)$ and $\bar{G}_2(x)$ are left coprime, it follows from Lemma \ref{coprime-reduction}
that there are polynomials $H_1(x),H_2(x)\in A[x]$ such that
$$H_1(x)F_1(x)+H_2(x)G_2(x)=1.$$
We can write
$$F_2(x)=H_1(x)F_1(x)F_2(x)+H_2(x)G_2(x)F_2(x)=$$
$$(H_1(x)G_1(x)+H_2(x)F_2(x))G_2(x)+H_2(x)(G_2(x)N(x)-N(x)G_2(x)).$$
\newline
Setting $K(x)=H_1(x)G_1(x)+H_2(x)F_2(x)$, we see that $F_2(x)=K(x)G_2(x)$ as elements of $(A/I^{(d+1)})[x]$, because $(A,I)$ is commutative with respect to its commutator filtration. Since $F_2(x)$ and $G_2(x)$ are monic polynomials of the same degree, we conclude that $F_2(x)=G_2(x)$ as  elements of $(A/I^{(d+1)})[x]$, that is, $N(x)=F_2(x)-G_2(x)\in I^{(d+1)}[x],$ a contradiction . 

\end{proof}

\end{subsection}
\begin{subsection}{Factorizations of polynomials over perfect pairs}
 In this part, we prove the following result.  
\begin{proposition}\label{factorizationN}
Let $(A,I)$ be a perfect pair and $F(x)\in A[x]$ be a monic polynomial. Suppose that $\bar{F}(x)$ has a factorization $\bar{F}(x)=f_1(x)f_2(x)$ over $A/I$ where $f_1(x),f_2(x)\in (A/I)[x]$ are left coprime monic polynomials. Then, there exists a perfect Jacobson pair
$(A\langle F;f_1,f_2\rangle,I\langle F;f_1,f_2\rangle)$ and a morphism
$$\Phi_{\langle F;f_1,f_2\rangle}:(A,I)\to  (A\langle F;f_1,f_2\rangle,I\langle F;f_1,f_2\rangle)$$ of pairs having the following universal property. For every morphism 
$$\phi:(A,I)\to (B,K)$$ of pairs, where $(B,K)$ is a perfect Jacobson pair,  if $\phi(F(x))=G_1(x)G_2(x)$ for some monic polynomials 
$G_1(x), G_2(x)\in B[x]$ such that
$$\bar{G}_1(x)=\bar{\phi}(f_1(x)), \bar{G}_2(x)=\bar{\phi}(f_2(x)),$$
then there exists a unique morphism 
$$\psi:(A\langle F;f_1,f_2\rangle,I\langle F;f_1,f_2\rangle)\to (B,K)$$ of pairs such that $\phi=\psi\circ \Phi_{\langle F;f_1,f_2\rangle}$. 
\end{proposition}
\begin{proof}
First, we give a construction of the pair $(A\langle F;f_1,f_2\rangle,I\langle F;f_1,f_2\rangle)$ after which we prove its universal property. 
Let $$F(x)=a_0+a_1x+\cdots+a_{d-1}x^{d-1}+x^d,$$ 
$$f_1(x)=b_0+b_1x+\cdots+b_{d_1-1}x^{d_1-1}+x^{d_1},$$
$$f_2(x)=c_0+c_1x+\cdots+c_{d_2-1}x^{d_2-1}+x^{d_2}.$$
We consider the free $A$-ring
$A\langle y_0,...,y_{d_1-1},z_0,...,z_{d_2-1}\rangle$  generated by the noncommutating variables 
$y_0,...,y_{d_1-1},z_0,...,z_{d_2-1}$.  We have 
$$(y_0+y_{1}x+\cdots+y_{d_1-1}x^{d_1-1}+x^{d_1})(z_0+z_{1}x+\cdots+z_{d_2-1}x^{d_2-1}+x^{d_2})=$$
$$g_0+g_1x+\cdots+
g_{d-1}x^{d-1}+x^d$$
where $g_0,...,g_{d-1}\in A\langle y_0,...,y_{d_1-1},z_0,...,z_{d_2-1}\rangle$.  
Consider the ring homomorphism
$$\alpha:A\langle y_0,...,y_{d_1-1},z_0,...,z_{d_2-1}\rangle\to \frac{A}{I}$$ defined by
$$\alpha(y_0)=b_0,...,\alpha(y_{d_1-1})=b_{d_1-1}, \alpha(z_0)=c_0,...,\alpha(z_{d_2-1})=c_{d_2-1},$$
$$\alpha(a)=\bar{a}\, \;\text{where}\; a\in A.$$
The ideal  $\langle g_0-a_0,...,g_{d-1}-a_{d-1}\rangle$ generated by $g_0-a_0,...,g_{d-1}-a_{d-1}$ is contained in the kernel of $\alpha$
because $\bar{F}(x)=f_1(x)f_2(x)$. Therefore, $\alpha$ gives rise to a  ring homomorphism 
$$\beta:\frac{A\langle y_0,...,y_{d_1-1},z_0,...,z_{d_2-1}\rangle}{\langle g_0-a_0,...,g_{d-1}-a_{d-1}\rangle}\to \frac{A}{I}$$
Consider the following localization ring (see subsection \ref{localizationsection})
$$R= \Big{(}
\frac{A\langle y_0,...,y_{d_1-1},z_0,...,z_{d_2-1}\rangle}{\langle g_0-a_0,...,g_{d-1}-a_{d-1}\rangle}\Big{)}_\beta.$$
We have canonical ring homomorphisms $\gamma:R\to  \frac{A}{I}$
and  $\eta:A\to R$
which satisfy $\gamma(\eta(a))=\bar{a}$ for every $a\in A$, see Proposition \ref{localization}. We set $J=\ker(\gamma)$. 
By Proposition \ref{localization}, $\gamma$ is local from which it follows that $J\subset  rad (A\langle F;f_1,f_2\rangle)$, that is, $(R,J)$ is a Jacobson pair. 
Since the quotient homomorphism $A\to A/I$ is a morphism $(A,I)\to (A/I,0)$ of pairs, we can use Lemma \ref{morphism} to deduce that 
$\eta:(A,I)\to (R,J)$ is a morphism of pairs.
Finally, we set
$$(A\langle F;f_1,f_2\rangle,I\langle F;f_1,f_2\rangle)=\digamma{(R,J)}$$
The canonical morphism $\eta:(A,I)\to (R,J)$ composed with the quotient morphism $(R,I)\to \digamma{(R,J)}$ gives a morphism 
$$\Phi_{\langle F;f_1,f_2\rangle}:(A,I)\to (A\langle F;f_1,f_2\rangle,I\langle F;f_1,f_2\rangle)$$
of pairs. 

To prove the universal property of $\Phi_{\langle F;f_1,f_2\rangle}$, let $\phi:(A,I)\to (B,K)$ be a morphism of pairs where $(B,K)$ is a perfect Jacobson pair.  Suppose that $\phi(F)(x)=G_1(x)G_2(x)$ where $G_1(x), G_2(x)\in B[x]$ are monic polynomials and 
$$\bar{G}_1=\bar{\phi}(f_1), \bar{G}_2=\bar{\phi}(f_2).$$
Let $$G_1(x)=e_0+e_1x+\cdots+e_{d_1-1}x^{d_1-1}+x^{d_1},$$
$$G_2(x)=f_0+f_1x+\cdots+f_{d_2-1}x^{d_2-1}+x^{d_2}.$$
One can easily check that the assignments 
$$y_0\mapsto e_0,...,y_{d_1-1}\mapsto e_{d_1-1}, z_0\mapsto f_0,...,z_{d_2-1}\mapsto f_{d_2-1}$$
yield a ring homomorphism 
$$\psi_1:\frac{A\langle y_0,...,y_{d_1-1},z_0,...,z_{d_2-1}\rangle}{\langle g_0-a_0,...,g_{d-1}-a_{d-1}\rangle}\to B$$
which extends the ring homomorphism $\phi$. 
Furthermore, since $K\subset rad(B)$ and $(B,K)$ is perfect, one can extend $\psi_1$ to a ring homomorphism
$$\psi:A\langle F;f_1,f_2\rangle\to B$$
satisfying $\phi=\psi\circ \Phi_{\langle F;f_1,f_2\rangle}$ as ring homomorphisms. 
Since $\bar{G}_1(x)=\bar{\phi}(f_1(x)), \bar{G}_2=\bar{\phi}(f_2(x))$, the following diagram is commutative
$$\begin{array}{ccc}
A\langle F;f_1,f_2\rangle & \xrightarrow{\psi} & B \\
\downarrow{} &  & \downarrow  \\
A/I & \xrightarrow{\bar{\phi}} & B/K 
\end{array}$$
Using the commutativity of this diagram and Lemma \ref{morphism}, we conclude that 
$$\psi:(A\langle F;f_1,f_2\rangle, I\langle F;f_1,f_2\rangle)\to (B,K)$$
is, in fact, a morphism of pairs. 
The uniqueness of $\psi$ follows from the fact that $(B,K)$ is a UFP by Proposition \ref{uniquenesslifting}. 
\end{proof}

\end{subsection} 
\begin{subsection}{LF-extensions}
Let $(A,I)$ be a perfect pair. A morphism $\phi:(A,I)\to (B,J)$ of pairs is called a \textit{simple left factorization extension} (or simple \textit{LF-extension}  for short) of $(A,I)$ if $\phi=\Phi_{\langle F;f_1,f_2\rangle }$ for some polynomials $F(x)\in A[x]$, $f_1(x),f_2(x)\in 
(A/I)[x]$ satisfying the conditions in Proposition \ref{factorizationN}. An \textit{LF-extension} of $(A,I)$ is a morphism $\phi:(A,I)\to (B,J)$ of pairs which is obtained by a finite sequence of simple LF-extensions, that is, there are simple LF-extensions 
$$\phi_i:(A_i,I_i)\to (A_{i+1},I_{i+1}), \, \text{where}\; i=1,...,d,$$ such that $(A_1,I_1)=(A,I)$, $(A_{d+1},I_{d+1})=(B,J)$ and $\phi=\phi_{d}\circ\phi_{d-1}\circ\cdots\circ\phi_1.$

Obviously, the collection of all LF-extensions of a perfect pair $(A,I)$ is a set which we denote by $LFext(A,I)$. Given morphisms 
$$\phi_1:(A,I)\to (B_1,J_1),\; \text{and}\; \phi_2:(A,I)\to (B_2,J_2)$$ in $LFext(A,I)$, we write $\phi_1\leq \phi_2$ if there exists a morphism 
$$\psi: (B_1,J_1)\to (B_2,J_2)$$ of pairs such that $\phi_2=\psi\circ\phi_1$.  Clearly, the relation $\leq$ defines a partial order on $LFext(A)$. Furthermore, we have the following result. 
\begin{lemma}\label{directedN}
Let $(A,I)$ be a perfect pair. (i) For all $\phi_1:(A,I)\to (B_1,J_1)$ and  $\phi_2:(A,I)\to (B_2,J_2)$ in $LFext(A)$, there exists at most one morphism $\psi: (B_1,J_1)\to (B_2,J_2)$ of pairs such that $\phi_2=\psi\circ \phi_1$. (ii) The partial order $\leq$ on $LFext(A)$ is directed. 
\end{lemma}
\begin{proof}
Part (i) can be proved using induction and the fact that $(B_2,J_2)$ is a UFP. To prove (ii), we first assume that 
$\phi_1:(A,I)\to (B_1,J_1)$ is a simple LF-extension. Therefore, $\phi=\Phi_{\langle F;f_1,f_2\rangle }$ for some polynomials $F(x)\in A[x]$, $f_1(x),f_2(x)\in (A/I)[x]$ satisfying the conditions in Proposition \ref{factorizationN}. Let $G_1(x)=\phi_2(F(x))$,
$g_1(x)=\bar{\phi}_2(F(x))$ and $g_2(x)=\bar{\phi}_2(F(x))$. It is easy to see that the polynomials $G_1(x),g_1(x),g_2(x)$ satisfy the conditions in Proposition \ref{factorizationN}, giving rise to a morphism 
$$\Phi_{\langle G;g_1,g_2\rangle}:(B_2,J_2)\to (B\langle G;g_1,g_2\rangle,I\langle G;g_1,g_2\rangle)$$
of pairs. Clearly 
$$\Phi_{\langle G;g_1,g_2\rangle}\circ\phi_2:(A,I)\to (B\langle G;g_1,g_2\rangle,I\langle G;g_1,g_2\rangle)$$
is an LF-extension. By the universal property of $\phi=\Phi_{\langle F;f_1,f_2\rangle }$, we see that there exists a morphism 
$$\psi:(B_1,J_1)\to (B\langle G;g_1,g_2\rangle,I\langle G;g_1,g_2\rangle)$$
such that $\Phi_{\langle G;g_1,g_2\rangle}\circ\phi_2=\psi\circ\phi_1$. It follows that 
$$\phi_1\leq \Phi_{\langle G;g_1,g_2\rangle}\circ\phi_2\;\text{and} \; \phi_2\leq \Phi_{\langle G;g_1,g_2\rangle}\circ\phi_2.$$
The general case is proved by induction. 
\end{proof}
\end{subsection}
\begin{subsection}{Left Henselizations}
In this subsection, we prove the following result concerning the concept of Henselization.  
\begin{theorem}\label{NHenselCom}
Let $(A,I)$ be a perfect pair.  Then, there exists a left Henselian 
pair $(A^{lh},I^{lh})$ and a morphism 
$\phi^{lh}:(A,I)\to (A^{lh},I^{lh})$ of pairs having the following universal property. For every morphism $\phi:(A,I)\to (B,J)$ of pairs from $(A,I)$ to a left Henselian prefect pair $(B,J)$, there exists a unique morphism 
$\psi:(A^{lh},I^{lh})\to (B,J)$ of pairs such that $\phi=\psi \circ\phi^{lh}$. 
\end{theorem}
\begin{proof}
By Lemma \ref{directedN}, the direct limit $(A^{lh},I^{lh})$ of elements in $LFext(A)$ exists. Moreover, it is a perfect pair. 
We also have a canonical morphism 
$$\phi^{lh}:(A,I)\to (A^{lh},I^{lh})$$ of pairs. The universal property of $\phi^{h}$ follows from Proposition 
\ref{factorizationN} and properties of direct limits. So, it remains to show that $(A^{lh},I^{lh})$ is left Henselian. Since 
$(A^{lh},I^{lh})$ is a direct limit of Jacobson pairs, it is a Jacobson pair.  Let a monic polynomial $F(x)\in A^{lh}[x]$ be given such that $\bar{F}(x)=f_1(x)f_2(x)$ for some left coprime
monic polynomials  $f_1(x),f_2(x)\in (A^{lh}/I^{lh})[x]$. Since $F(x)$ has only finitely many (nonzero) coefficients, there exists an LF-extension $\phi:(A,I)\to (B,J)$ such that $F(X)\in B[X]$. By Proposition \ref{factorizationN}, the polynomial $F(X)$ has a factorization $F(x)=F_1(x)F_2(x)$ over $B(F;f_1,f_2)$, hence over $A^{lh}$, such that $F_1,F_2\in A^{lh}[X]$ are monic, and $\bar{F}_1(x)=f_1(x)$, $\bar{F}(x)=f_2(x)$. Since $(A^{lh},I^{lh})$ is a perfect Jacobson ring, it is a UFP, see Proposition \ref{uniquenesslifting}. Therefore,  the factorization $F(x)=F_1(x)F_2(x)$ is unique. It follows that $(A^{lh},I^{lh})$ is a left Henselian pair, and we are done.   
\end{proof}
The pair $(A^{lh},I^{lh})$  is called the \textit{left Henselization} of $(A,I)$. It is easy to see that the pair $(A^{lh},I^{lh})$ is unique up to unique isomorphism. Similarly, one can show that 
every perfect pair $(A,I)$ has a right Henselizaiton $$\phi^{rh}:(A,I)\to (A^{rh},I^{rh})$$ satisfying the corresponding universal property. 
We note that if $A/I$ is, in addition, a commutative ring, then the right Henselization of $(A,I)$ is also the left Henselization of $(A,I)$, and vice versa. 
\end{subsection}
\begin{subsection}{Commutative Henselizations}
A pair $(A,I)$ is called \text{commutative} if $A$ is a commutative ring. It is known that every commutative pair has a Henselization in the category $\mathcal{P}_c$ of commutative pairs, see \cite{lafon1963anneaux}. The Henselization of a commutative pair $(A,I)$ in $\mathcal{P}_c$ is referred to as the commutative Henselization  of $(A,I)$, and is denoted by $(A^{ch},I^{ch})$.
Obviously, the category $\mathcal{P}_c$  is  a full subcategory of $\mathcal{P}_0$.
The following result determines commutative Henselizations in terms of left Henselizations.
\begin{proposition}\label{NHenselCom}
Let $(A,I)$ be a commutative pair.  Then, the ideal $J$ generated by $[A^{lh},A^{lh}]$ is contained in $I^{lh}$. Moreover, the morphism 
$$q\circ \phi^{lh}:(A,I)\to (\frac{A}{J}^{lh},\frac{I}{J}^{lh}),$$
where $q:(A^{lh},I^{lh})\to (A^{lh}/J,I^{lh}/J)$ is the quotient morphism, is the commutative Henselization of $(A,I)$.  
\end{proposition}
\begin{proof}
Let $(A^{ch},I^{ch})$  be the commutative Henselization of $(A,I)$ and 
$$\phi^{ch}:(A,I)\to (A^{ch},I^{ch}),$$ be the corresponding morphism of pairs. Since $(A^{ch},I^{ch})$ is left Henselian, there exists a unique morphism 
$$\phi:(A^{lh},I^{lh})\to (A^{ch},I^{ch}),$$
of pairs such that $\phi^{ch}=\phi\circ\phi^{lh}$.  The fact that $(A^{ch},I^{ch})$ is commutative implies that 
the ideal $J$ generated by $[A^{lh},A^{lh}]$ is contained in the ideal
$$\ker(\phi)=\phi^{-1}(0)\subset \phi^{-1}(I^{ch})=I^{lh}.$$
Moreover,  there exists a unique morphism 
$$\psi:(\frac{A}{J}^{lh},\frac{I}{J}^{lh})\to (A^{ch},I^{ch})$$ such that $\phi=\psi\circ q$. Using the universal properties of $\phi^{lh}$ and $\phi^{ch}$, one can verify that $\psi$ is an isomorphism and the morphism 
$$q\circ \phi^{lh}:(A,I)\to (\frac{A}{J}^{lh},\frac{I}{J}^{lh})$$
is the commutative Henselization of $(A,I)$. 
\end{proof}

\end{subsection}
\begin{subsection}{Henselizations of local rings}\label{localringHensel}
We conclude this article with a discussion of Henselizations of local rings. A ring $A$ is called \textit{local} if the set of all nonunits in $A$ form an ideal. Every local ring $A$ has a unique maximal ideal $I$. A pair $(A,I)$ is called a local pair if $A$ is a local ring and $I$ is its maximal ideal. We note that any local pair is a Jacobson pair. 

\begin{proposition}\label{NHensellocal}
The left (right) Henselization of any perfect local pair is a local pair.   
\end{proposition}
\begin{proof}
Since the direct limit of local pairs is a local pair, it is enough to show that any simple LF-extension of a perfect local pair is a local pair. Let 
$$\Phi_{\langle F;f_1,f_2\rangle}:(A,I)\to  (A\langle F;f_1,f_2\rangle,I\langle F;f_1,f_2\rangle)$$
be a simple LF-extension where $(A,I)$ is a local ring.  Referring to the notations used in Proposition \ref{factorizationN}, one can see that the ring homomorphism $\gamma:R\to  A/I$ is a local homomorphism. Since $A/I$ is a local ring, Lemma \ref{localmapsLocalrings}
implies that $R$ is a local ring whose maximal ideal is $J=\gamma^{-1}(0)$. Using the relation
$$(A\langle F;f_1,f_2\rangle,I\langle F;f_1,f_2\rangle)=\digamma{(R,J)},$$
we conclude that $(A\langle F;f_1,f_2\rangle,I\langle F;f_1,f_2\rangle)$ is a local pair, and we are done. 
\end{proof}

\end{subsection}

\end{section}

\bibliographystyle{plain}
\bibliography{NHLbiblan}

 \end{document}